\documentclass[11pt,leqno]{article}
\usepackage{amssymb,amsmath,amsthm}
\usepackage[all,2cell,ps]{xy}
\usepackage[alphabetic,lite]{amsrefs}
\numberwithin{equation}{section}

\usepackage[top=2cm,bottom=2cm,left=3cm,right=3cm,marginparsep=0.3cm,marginparwidth=3cm,includefoot]{geometry}

\usepackage[colorlinks=true, pdfstartview=FitV, linkcolor=blue, citecolor=blue, urlcolor=blue]{hyperref}

\usepackage{enumerate}
\usepackage{mathrsfs}

\numberwithin{equation}{section}

\usepackage{enumerate}
\usepackage{mathrsfs}

\newcommand{\spa}{\vspace{0.5ex}\noindent}

\newenvironment{rouge}
{\relax\color{red}}
{\hspace*{.3ex}\relax}
\newcommand{\ber}{\begin{rouge}}
\newcommand{\er}{\end{rouge}}

\newcommand{\nc}{\newcommand}
\nc{\on}{\operatorname}
\newtheorem{theorem}{Theorem}[section]
\newtheorem{proposition}[theorem]{Proposition}
\newtheorem{lemma}[theorem]{Lemma}
\newtheorem{corollary}[theorem]{Corollary}

\theoremstyle{definition}

\newtheorem{definition}[theorem]{Definition}
\newtheorem{notation}[theorem]{Notation}
\newtheorem{example}[theorem]{Example}

\nc{\Prop}{\begin{proposition}}
\nc{\enprop}{\end{proposition}}
\nc{\Lemma}{\begin{lemma}}
\nc{\enlemma}{\end{lemma}}
\nc{\Cor}{\begin{corollary}}
\nc{\encor}{\end{corollary}}
\nc{\Def}{\begin{definition}}
\nc{\enDef}{\end{definition}}
\nc{\Th}{\begin{theorem}}
\nc{\entheorem}{\end{theorem}}

\newcommand{\C}{{\mathbb{C}}}

\newcommand{\R}{{\mathbb{R}}}

\newcommand{\Z}{{\mathbb{Z}}}

\newcommand{\cor}{{\bf k}}

\def\phi{{\varphi}}
\def\epsilon{\varepsilon}
\def\muhom{\mu hom}

\def\sha{\mathscr{A}}

\def\shb{\mathscr{B}}
\def\shc{\mathscr{C}}
\def\shd{\mathscr{D}}
\def\she{\mathscr{E}}

\def\shi{\mathscr{I}}

\def\shl{\mathscr{L}}
\def\shm{\mathscr{M}}

\def\sho{\mathscr{O}}

\def\shr{\mathscr{R}}
\def\shs{\mathscr{S}}

\newcommand{\ol}{\overline}

\newcommand{\lp}{{\rm(}}
\newcommand{\rp}{{\rm)}}

\nc{\RR}{\mathrm{R}}
\nc{\LL}{\mathrm{L}}

\newcommand{\into}{\hookrightarrow}

\nc{\CH}{\on{H}}

\newcommand{\HWo}[1][X]{\widehat{\mathscr{W}}_{#1}(0)}

\newcommand{\A}[1][X]{\mathscr{A}_{{#1}}}
\newcommand{\Ah}[1][X]{\mathscr{A}_{{#1}}^{\rm loc}}
\nc{\DA}[1][X]{\mathscr{D}^{\mathscr{A}}_{#1}}
\nc{\Dh}[1][X]{\mathscr{D}_{#1}[[\hbar]]}
\nc{\Dhh}[1][X]{\mathscr{D}_{#1}[\hbar]}
\nc{\Dhhh}[1][X]{\mathscr{D}_{#1}[\hbar,\opb{\hbar}]}

\nc{\Dhl}[1][X]{\shd_{#1}\Ls}
\nc{\Dho}[1][X]{\mathscr{D}_{#1}^\hbar(0)}
\nc{\Oh}[1][X]{\sho_{#1}^\hbar}
\nc{\OOh}[1][X]{\sho_{#1}[[\hbar]]}
\nc{\oh}[1][X]{\Omega_{#1}[[\hbar]]}

\nc{\Ohl}[1][X]{\sho_{#1}^{\hbar,\loc}}
\nc{\ohh}[1][X]{\Omega_{#1}^\hbar}

\newcommand{\gr}{\mathrm{gr}_\hbar}

\newcommand{\Lie}[1][]{\operatorname{\mathsf{L}}\def\temp{#1}
\ifx\temp\empty\else^{(#1)}\fi}

\newcommand{\stkHom}[1][]{\mathfrak{Hom}_{\raise1.5ex\hbox to.1em{}#1}}

\newcommand{\Int}{{\rm Int}}
\newcommand{\loc}{{\rm loc}}

\renewcommand{\to}[1][]{\xrightarrow[]{#1}}

\newcommand{\isoto}[1][]{\xrightarrow[#1]%
{{\raisebox{-.6ex}[0ex][-.6ex]{$\mspace{1mu}\sim\mspace{2mu}$}}}}

\newcommand{\Hom}[1][]{\mathrm{Hom}_{\raise1.5ex\hbox to.1em{}#1}}
\newcommand{\RHom}[1][]{\RR\mathrm{Hom}_{\raise1.5ex\hbox to.1em{}#1}}
\newcommand{\Ext}[2][]{\mathrm{Ext}_{\raise1.5ex\hbox to.1em{}#1}^{#2}}
\renewcommand{\hom}[1][]{{\mathscr{H}\mspace{-4mu}om}_{\raise1.5ex\hbox to.1em{}#1}}
\newcommand{\rhom}[1][]{{\RR\mathscr{H}\mspace{-3mu}om}_{\raise1.5ex\hbox to.1em{}#1}}
\newcommand{\ext}[2][]{{\mathscr{E}\mspace{-2mu}xt}_{%
\raise1.5ex\hbox to.1em{}#1}^{#2}}
\def\endo{{\she \kern-.2ex nd}}

\newcommand{\Tens}[1][]{\mathbin{\otimes_{\raise1.5ex\hbox to-.1em{}{#1}}}}
\newcommand{\LTens}[1][]{\mathbin{\otimes_{\raise1.5ex\hbox to-.1em{}#1}^{L}}}
\newcommand{\Tor}[2][]{\mathrm{Tor}^{\raise1.5ex\hbox to.1em{}#1}_{#2}}
\newcommand{\tens}[1][]{\mathbin{\otimes_{\raise1.5ex\hbox to-.1em{}{#1}}}}

\newcommand{\dtens}[1][]%
{{\overset{\mathrm{L}}{\underline{\otimes}}}_{#1}}

\nc{\aut}{{\sha\mspace{-1mu}\mit{ut}}\,}

\newcommand{\Endo}[1][]{\mathrm{End}_{\raise1.5ex\hbox to.1em{}#1}}

\newcommand{\Aut}[1][]{\mathrm{Aut}_{\raise1.5ex\hbox to.1em{}#1}}
\newcommand{\sect}{\Gamma}
\newcommand{\rsect}{\mathrm{R}\Gamma}

\newcommand{\Rb}{{\rm b}}

\newcommand{\SSi}{\on{SS}}
\DeclareMathOperator{\musupp}{\mu supp}
\newcommand{\oim}[1]{{#1}_*}
\newcommand{\eim}[1]{{#1}_!}
\newcommand{\roim}[1]{\RR{#1}_*}
\newcommand{\reim}[1]{\RR{#1}_!}
\newcommand{\opb}[1]{#1^{-1}}
\newcommand{\epb}[1]{#1^{!}}

\newcommand{\dopb}[1]{#1^{D}}

\newcommand{\tw}[1]{\widetilde{#1}}

\DeclareMathOperator{\supp}{supp}
\DeclareMathOperator{\ori}{or}
\DeclareMathOperator{\chv}{char}

\DeclareMathOperator{\hchv}{hypchar}

\def\rop{{\rm op}}

\newcommand{\Pro}{\mathrm{Pro}}

\newcommand{\eqdot}{\mathbin{:=}}

\newcommand{\cl}{\colon}
\newcommand{\scbul}{{\,\raise.4ex\hbox{$\scriptscriptstyle\bullet$}\,}}

\def\suba#1{{\raise-.2em\hbox{${\raise .2em\hbox{$\alpha$}}_{#1}$}}}

\newcommand{\ba}{\begin{array}}
\newcommand{\ea}{\end{array}}

\newcommand{\bnum}{\begin{enumerate}[{\rm(i)}]}
\newcommand{\enum}{\end{enumerate}}
\newcommand{\banum}{\begin{enumerate}[{\rm(a)}]}
\newcommand{\eanum}{\end{enumerate}}

\newcommand{\eq}{\begin{eqnarray}}
\newcommand{\eneq}{\end{eqnarray}}
\newcommand{\eqn}{\begin{eqnarray*}}
\newcommand{\eneqn}{\end{eqnarray*}}

\nc{\Der}{\on{D}}

\nc{\Proof}{\begin{proof}}
\nc{\QED}{\end{proof}}
\nc{\bA}[1][X]{\gr({\sha}_{#1})}
\nc{\OA}[1][X]{\Omega_{#1}^{\,\mathscr{A}}}
\nc{\oA}[1][X]{\omega_{#1}^{\,\mathscr{A}}}
\nc{\omA}[1][X]{\omega_{#1}}
\nc{\oAI}[1][X]{\omega_{#1}^{\,\mathscr{A}{\otimes-1}}}
\nc{\omAI}[1][X]{\omega_{#1}^{{\otimes-1}}}
\nc{\oAh}[1][X]{\omega_{#1}^{\,\Ah[]}}
\nc{\oo}[1][X]{\omega_{#1}}
\nc{\Derb}{\mathrm{D}^{\mathrm{b}}}
\nc{\hs}{\hspace*}
\nc{\Supp}{\on{Supp}}
\nc{\tr}{\on{tr}}

\newcommand{\HHWo}[1][X]{\mathcal{HH}(\HWo[])}

\newcommand{\RD}{{\rm D}}
\nc{\RDAA}[1][X]{\mathrm{D}^\prime_{\mathscr{A}_{#1}}}
\nc{\RDA}{\mathrm{D}^\prime_{\mathscr{A}}}
\nc{\RDAh}{\mathrm{D}^\prime_{\mathscr{A}^{\rm loc}}}
\nc{\RDO}{\mathrm{D}^\prime_{\mathscr{O}}}
\nc{\RDOl}{\mathrm{D}^\prime_{\mathscr{O}^{\hbar,\rm loc}}}
\nc{\RDDO}{\mathrm{D}_{\mathscr{O}}}
\nc{\RDDA}{\mathrm{D}_{\mathscr{A}}}
\nc{\RDD}{\mathrm{D}^\prime}
\nc{\conv}[1][]{\mathop{\circ}\limits_{#1}}
\nc{\sconv}[1][]{\mathop{\ast}\limits_{#1}}

\nc{\ssum}{\mathop{\mbox{\normalsize$\sum$}}}
\nc{\de}[1][X]{\delta_{#1}} %diagonal embedding
\nc{\vs}{\vspace}
\nc{\dA}[1][X]{\mathscr{C}_{{#1}}}
\nc{\dO}[1][X]{{{\delta_{{#1}}}_*\mspace{1mu}\mathscr{O}_{{#1}}}}
\nc{\dGA}[1][X]{\gr(\mathscr{C}_{{#1}})}

\nc{\OS}[1][X]{\sho_{#1}}
\nc{\soplus}{\mathop{\text{\scriptsize\raisebox{.5ex}{$\displaystyle\bigoplus$}}}}
\nc{\Inv}{\on{Inv}}
\nc{\stkInv}{\mathfrak{Inv}}
\nc{\bwr}{{\mbox{\large$\wr$}}}
\nc{\forl}{[\mspace{-.3mu}[\hbar]\mspace{-.3mu}]}
\nc{\Ls}{(\mspace{-.3mu}(\hbar)\mspace{-.3mu})}

\nc{\be}{\begin{enumerate}}
\nc{\ee}{\end{enumerate}}
\nc{\stan}{\mathrm{stan}}
\nc{\Db}{\RD^\Rb}
\nc{\pt}{\mathrm{pt}}
\nc{\BBD}{\mathbb{D}}
\nc{\rC}{\mathrm{C}}
\nc{\scup}{\mathop{\text{\scriptsize\raisebox{.5ex}{$\displaystyle\bigcup$}}}}
\nc{\tX}{{\widetilde{X}}}
\nc{\AL}{\A[\Lambda/X]}
\nc{\ALa}{\A[\Lambda^a]}
\nc{\Gr}{\on{Gr}}
\nc{\CL}{\on{\mathrm{C}_\Lambda}}
\nc{\codim}{\on{codim}}
\nc{\Chl}{\on{char_{\Lambda}}}
\nc{\Chlo}{\on{char_{\Lambda_0}}}
\nc{\ChM}{\on{char_{M}}}
\nc{\DL}{\shd_\shl}
\nc{\DLl}{\shd_\shl^\loc}
\nc{\rmC}{\rm C}
\nc{\spec}{\rm spec}
\nc{\WF}{\rm WF}

\nc{\Zh}{\Z\forl}
\nc{\PZ}{\Pro(\Zh)}
\nc{\coco}{{cohomologically complete}}
\nc{\rpi}{{{\rm R}\pi}}
\nc{\shal}{{\sha^\loc}}

\nc{\Ran}{{\rm Ran}\, X}

\nc{\cc}{cohomologically complete}
\nc{\shrl}{{\shr^\loc}}
\newcommand{\Sol}{{\shs\mspace{-2.5mu}\mathit{ol}}}

\begin{document}

\title
{Wick rotation for  D-modules}
%\date{}
\author{Pierre Schapira}
\maketitle

\begin{abstract}
We extend the classical Wick rotation to D-modules and higher codimensional submanifolds.
\end{abstract}
{\renewcommand{\thefootnote}{\mbox{}}
\footnote{Research supported by the ANR-15-CE40-0007 ``MICROLOCAL''.}
%\addtocounter{footnote}{-2}
}

\section{Introduction}
Let $M$ be a real analytic manifold of the type $N\times\R$ and let  $X=Y\times\C$ be a complexification of $M$.  Consider a differential operator $P$ on $X$ such that $P$ is hyperbolic on $M$ with respect to the direction $N\times\{0\}$, a typical example being  the wave operator on a spacetime. 
Denote by $L$ the real manifold $N\times\sqrt{-1}\R$. It may happen, and it happens for the wave operator, that $P$ is elliptic on $L$. Passing from $M$ to $L$ is called the Wick rotation by physicists who deduce interesting properties of $P$ on $M$ from the study of $P$ on  $L$. 

In the situation  above, we had $codim_M N = codim_L N = 1$. In this paper, we treat the general case of two real analytic manifolds $M$ and $L$ in $X$, $X$ being a complexification of both $M$ and $L$, such that the intersection $N\eqdot M\cap L$ is clean, 
and we consider a coherent $\shd_X$-module $\shm$ which is hyperbolic with respect to $M$ on $N$ and elliptic on $L$. The main result is Theorem~\ref{th:wick} which describes an isomorphism  in a neighborhood of $N$
between the complex of hyperfunction solutions of $\shm$ on $L$ defined in a given cone  $\gamma\subset T_NL$ and  the complex of  hyperfunction solutions of $\shm$ on $M$ with wave front set in a  cone $\lambda\subset T^*_MX$ associated with $\gamma$. It is also proved that this isomorphism is compatible with the boundary values morphism from $M$ to $N$ and from $L$ to $N$.

\spa
{\bf Aknowledgements} This paper was initiated by a series of discussions with Christian G{\'e}rard who kindly explained us some 
problems associated with the classical Wick rotation (see~\cite{GW17}). We sincerely thank him for his patience and his explanations. 

\section{Sheaves, D-modules and wave front sets}
\subsection{Sheaves}
We shall use the microlocal theory of sheaves of~\cite{KS90} and mainly follow its terminology. For the reader's convenience, we recall a few notations and results. 

\subsubsection*{Geometry}
Let $X$ be a real manifold of class $C^\infty$. For  a subset $A\subset X$, we denote by $\ol A$ its closure and by $\Int(A)$ its interior. 
We denote by 
\eqn
&& \tau_X\cl TX\to X, \quad \pi_X\cl T^*X\to X
\eneqn
the tangent bundle and the cotangent bundle to $X$. 
For a closed submanifold $M$ of $X$, we denote by $\tau_M\cl T_MX\to M$ and $\pi_M\cl T^*_MX\to M$ the normal bundle and the conormal bundle to $M$ in $X$. In particular, $T^*_XX$ is the zero-section of $T^*X$, that we identify with $X$.  

For a vector bundle $\pi\cl E\to X$, we identify $X$ with the zero-section, we denote by $E_x$ the fiber of $E$ at $x\in X$, we set $\dot E=E\setminus X$ and we denote by 
$\dot\pi\cl \dot E\to X$ the projection.
For a cone $\gamma$ in a vector bundle $E\to X$, we set $\gamma_x=\gamma\cap E_x$, we denote by $\gamma^a=-\gamma$ the opposite cone and by $\gamma^\circ$ the polar cone in the dual vector bundle $E^*$,
\eqn
&&\gamma^\circ=\{(x;\xi)\in E^*;\langle \xi,v\rangle\geq0 \mbox{ for all $x\in M, v\in \gamma_x$}\}.
\eneqn
For $A\subset X$, the Whitney normal cone  of $A$ along $M$, $C_M(A)\subset T_MX$, is defined in~\cite{KS90}*{Def.~4.1.1}.

To a morphism of manifolds $f\cl Y\to X$, one associates the maps:
\eq\label{diag:microlocal1}
&&\xymatrix{
T^*Y\ar[dr]_-{\pi_Y}&\ar[l]_-{f_d}Y\times_XT^*X\ar[r]^-{f_\pi}\ar[d]^-\pi&T^*X\ar[d]^-{\pi_X}\\
&Y\ar[r]^-f&X
}\eneq
where $f_d$ is the transpose of the tangent map  $Tf\cl TY\to Y\times_XTX$.
\begin{definition}
Let $\Lambda$ be a closed conic subset of $T^*X$. One says that $f$ is non characteristic for $\Lambda$ if the map $f_d$ is proper on $\opb{f_\pi}(\Lambda)$.
\end{definition}
\subsubsection*{Sheaves}
Let $\cor$ be a field. 
One denotes by $\Derb(\cor_X)$ the bounded derived category of sheaves of $\cor$-vector spaces on $X$. We simply call an object of this category ``a sheaf''. 
For a closed subset $A$ of a manifold
we denote by $\cor_A$ the constant sheaf on $A$ with stalk $\cor$ extended by $0$  outside of $A$.  More generally, we shall identify a sheaf on $A$ and its extension by $0$   outside of $A$. If $A$ is locally closed, we keep the notation $\cor_A$ as far as there is no risk of confusion. 
We denote by  $\omega_X$ the dualizing complex on $X$. Recall that $\omega_X\simeq\ori_X\,[\dim X]$ where $\ori_X$ is the orientation sheaf and $\dim X$ is the dimension of $X$. More generally, we consider the relative  dualizing complex associated with  a morphism $f\cl Y\to X$, $\omega_{Y/X}=\omega_Y\tens\opb{f}(\omega_X^{\otimes-1})$ and its inverse, $\omega_{X/Y}= \omega_{Y/X}^{\otimes-1}$.
We denote by 
$\RD'_X(\scbul)=\rhom(\scbul,\cor_X)$  the duality functor on $X$. 

We shall use freely the six Grothendieck operations on sheaves. 

\subsubsection*{Microlocalization}
For a closed submanifold $M$ of $X$, 
we have the functors 
\eqn
&&\nu_M\cl \Derb(\cor_X)\to\Derb_{\R^+}(\cor_{T_MX})\mbox{ specialization along $M$},\\
&&\mu_M\cl \Derb(\cor_X)\to\Derb_{\R^+}(\cor_{T^*_MX})\mbox{ microlocalization along $M$},\\
&&\muhom \cl \Derb(\cor_X)\times\Derb(\cor_X)^\rop\to\Derb_{\R^+}(\cor_{T^*X}).
\eneqn
Here, for a vector bundle $E\to M$ or $E\to X$, $\Derb_{\R^+}(\cor_{E})$ is the full subcategory of $\Derb(\cor_{E})$ consisting of conic sheaves, that is, sheaves locally constant under the $\R^+$-action. 

The functor $\mu_M$, called Sato's microlocalization functor,  is the  Fourier--Sato transform of the specialization functor $\nu_M$. The bifunctor $\muhom$ of~\cite{KS90} is  a  slight generalization of $\mu_M$. Recall that $\mu_M(\scbul)=\muhom(\cor_M,\scbul)$. 

Let $\lambda$ be a closed convex proper cone of $T^*_MX$ containing the zero-section $M$. For $F\in\Derb(\cor_X)$, we have an isomorphism (see~\cite{KS90}*{Th.~4.3.2}):
\eq\label{eq:nuandmu}
&&\roim{\pi_M}\rsect_\lambda(\mu_M(F))\tens\omega_{X/M}\simeq\roim{\tau_M}\rsect_{\Int(\lambda^{\circ a})}(\nu_M(F)).
\eneq
(Recall that $\lambda^{\circ a}$ is the opposite of the polar cone $\lambda^\circ$.)

\subsubsection*{Microsupport}
To a sheaf $F$ is associated its microsupport $\musupp(F)$\footnote{$\musupp(F)$ was denoted $\SSi(F)$ in~\cite{KS90}.}, a closed $\R^+$-conic {\em co-isotropic} subset of $T^*X$.

Let us recall some results that we shall use. 
\begin{theorem}\label{th:invim}
Let $f\cl Y\to X$ be a morphism of real manifolds and let $F\in\Derb(\cor_X)$. Assume that $f$ is non characteristic for $F$, that is, for $\musupp(F)$. Then the morphism $\opb{f}F\tens\,\omega_{Y/X}\to\epb{f}F$ is an isomorphim.
\end{theorem}
As a particular case of this result, we get a kind of  Petrowski theorem for sheaves (see Theorem~\ref{eq:elliptic1} below):
\begin{corollary}\label{cor:ell}
Let $M$ be a closed submanifold of $X$ and let $F\in\Derb(\cor_X)$.  Assume that $T^*_MX\cap\musupp(F)\subset T^*_XX$. Then 
$F\tens\cor_M\simeq \rsect_MF\tens\ori_{M/X}\,[\codim _XM]$.
\end{corollary}

Let $M$ be a closed submanifold of $X$. If $\Lambda\subset T^*X$ is a closed conic subset, its Whitney normal cone along $T^*_MX$ is a closed biconic subset of 
$T_{T^*_MX}T^*X\simeq T^*T^*_MX$. Moreover, there exists a natural embedding 
\eq\label{eq:hyp1}
&&T^*M\into T^*T^*_MX\simeq T_{T^*_MX}T^*X.
\eneq
Now we consider a morphism of manifolds $g\cl L\to X$ and let $M\subset X$ and $N\subset L$ be two closed submanifolds such that the map $g$ induces a closed embedding $g\vert_N\cl N\into M$. 
One gets the maps
\eq\label{diag:671}
&&\xymatrix{
T^*L&L\times_XT^*X\ar[l]_-{g_d}\ar[r]^-{g_\pi}&T^*X\\
T_N^*L\ar@{^(->}[u]&N\times_MT_M^*X\ar[l]_-{g_{Nd}}\ar[r]^-{g_{N\pi}}\ar@{^(->}[u]&T_M^*X.\ar@{^(->}[u]
}\eneq
The next result is a particular case of~\cite{KS90}*{Th.~6.7.1} in which we choose $V=T^*_NL$ and write 
$g\cl L\to X$ instead of $f\cl Y\to X$. (The reason of this change of notations is that we need to consider  the complexification of the embedding $N\into M$ that we shall denote by $f\cl Y\into X$.) 

\begin{theorem}\label{th:671}
Let $F\in\Derb(\cor_X)$ and assume
\banum
\item
$g$ is non characteristic for $F$,
\item
the map $g_{N\pi}$ is non characteristic for $C_{T^*_MX}(\musupp(F))$,
\item
$\opb{g_d}T_N^*L\cap\opb{g_\pi}\musupp(F)\subset L\times_XT^*_MX$.
\eanum
Then one has the commutative diagram of natural isomorphisms on $T^*_LX$:
\eq\label{eq:isowick1}
&&\xymatrix{
\reim{g_{Nd}}(\omega_{N/M}\tens\opb{g_{N\pi}}\mu_M(F))\ar[r]^-\sim\ar[d]^-\sim
      &\mu_N(\omega_{L/X}\tens\opb{g}F)\ar[d]^-\sim\\
\roim{g_{Nd}}(\epb{g_{N\pi}}\mu_M(F))&\ar[l]^-\sim\mu_N(\epb{g}F).
}\eneq
\end{theorem}

\begin{notation}
As usual, we have simply 
writen $\omega_M$ instead of $\opb{\pi}\omega_M$ and similarly with other locally constant sheaves. 
\end{notation}
Consider the projections
\eq\label{diag:5}
&&\xymatrix{
T_N^*L\ar[dr]_-{\pi_N}&N\times_MT_M^*X\ar[l]_-{g_{Nd}}\ar[r]^-{g_{N\pi}}\ar[d]^-\pi&T_M^*X\ar[d]^-{\pi_M}\\
&N\ar[r]^-g&M
}\eneq
One has the isomorphisms
\eq\label{eq:isolhs}
\roim{\pi_N}\roim{g_{Nd}}(\epb{g_{N\pi}}\mu_M(F))&\simeq &\roim{\pi}(\epb{g_{N\pi}}\mu_M(F))\nonumber\\
&\simeq &\epb{g}\roim{\pi_M}\mu_M(F)\simeq\rsect_NF,
\eneq
and 
\eq\label{eq:isorhs}
\roim{\pi_N}\mu_N(\epb{g}F)&\simeq &\rsect_N\epb{g}F\simeq\rsect_NF.
\eneq
Moreover, one easily proves:
\begin{lemma}\label{le:wickcompatible}
The isomorphisms~\eqref{eq:isolhs} and~\eqref{eq:isorhs} are compatible with the morphisms obtained by applying $\roim{\pi_N}$ 
to~\eqref{eq:isowick1}.
\end{lemma}

\begin{lemma}\label{le:671}
In the situation of {\rm Theorem~\ref{th:671}} assume moreover that $g\cl L\to X$ is a closed embedding, $N=L\cap M$ and the intersection is clean \lp that is, $TN=N\times_MTM\cap N\times_LTL$\rp. Then condition {\rm (c)} follows from {\rm (b)}. 
\end{lemma}

\begin{proof}
Let us choose a local coordinate system $(x',x'',y',y'')$ on $X$  such that $M=\{y'=y''=0)\}$ and $L=\{x''=y''=0\}$. Denote by 
$(x',x'',y',y'';\xi',\xi'',\eta',\eta'')$ the coordinates on $T^*X$ and   by  $(x',x'';\xi',\xi'')$ the coordinates on $T^*M$. Then
\begin{align*}
&M=\{y'=y''=0)\}, & T^*_MX=\{y'=y''=\xi'=\xi''=0\},\\
&L=\{x''=y''=0\}, & T^*_LX=\{x''=y''=\xi'=\eta'=0\},\\
&N=\{x''=y'=y''=0)\}, &T^*_NX=\{x''=y'=y''=\xi'=0\},\\
&g_d\cl (x',y';\xi',\xi'',\eta',\eta'')\mapsto (x',y';\xi',\eta'), &.
\end{align*}
Therefore $\opb{g_d}T_N^*L=\{ (x',y';\xi',\xi'',\eta',\eta'')\in L\times_XT^*X;y'=\xi'=0\}=T^*_NX$. 
Let $\theta\in T_{T^*_MX}T^*X$ with $\theta\notin C_{T^*_MX}(\musupp(F))$. Then 
$(x',x'';\eta',\eta'')+\theta\notin\musupp(F)$. Choosing $\theta\in T^*_NM$, $\theta\neq0$, we get that 
$(x',0;0,\xi'',\eta',\eta'')\in\musupp(F)$ implies $\xi''=0$. 
\end{proof}

\subsection{Analytic wave front set}
From now on and until the end of this paper, unless otherwise specified, all manifolds are (real or complex) analytic and the base field $\cor$ is $\C$.

Let $M$ be a real manifold of dimension $n$ and let $X$ be a complexification of $M$. 
One denotes by $\sha_M$ the sheaf of complex valued real analytic functions on $M$, that is, $\sha_M=\sho_X\vert_M$. 

One denotes by $\shb_M$ and $\shc_M$ the sheaves on $M$ and  $T^*_MX$ of Sato's hyperfunctions and microfunctions, respectively. 
 Recall that these sheaves are defined by
 \eqn
 &&\sha_M\eqdot\sho_X\tens\C_M,\quad \shb_M\eqdot\rhom(\RD'_X\C_M,\sho_X),\quad\shc_M\eqdot\muhom(\RD'_X\C_M,\sho_X).
 \eneqn
 In particular, $\rhom(\RD'_X\C_M,\sho_X)$ and $\muhom(\RD'_X\C_M,\sho_X)$ are concentrated in degree $0$. 
 Since $\RD'_X\C_M\simeq\ori_M\,[-n]\simeq\omega_{M/X}\simeq\omega_M^{\otimes-1}$,  we get that 
\eqn
&&\shb_M\simeq \rsect_M(\sho_X)\tens\omega_{M}\simeq H^n_M(\sho_X)\tens\ori_M,\\
&&\shc_M\simeq \mu_M(\sho_X)\tens\omega_{M}\simeq H^n(\mu_M(\sho_X))\tens\ori_M.
\eneqn
The sheaf $\shb_M$ is flabby and the sheaf $\shc_M$ is conically flabby. 

Moreover, since $\roim{\pi}\circ\muhom\simeq\rhom$, we have the isomorphism $\shb_M\isoto \oim{\pi}\shc_M$.
One deduces the isomorphism:
\eqn
&&\spec\cl \sect(M;\shb_M)\isoto\sect(T^*_MX;\shc_M).
\eneqn

\begin{definition}[\cite{Sa70}]
The analytic wave front set of a hyperfunction $u\in\sect(M;\shb_M)$, denoted $\WF(u)$, is the support of $\spec(u)$, a closed conic subset of $T^*_MX$. 
\end{definition}

The next result is well-known to the specialists.
Let $M$ be a real analytic manifold, $X$ a complexification of $M$ and let $\lambda$ be a closed convex proper cone in $T^*_MX$. 
\begin{theorem}\label{th:uniquec}
Let $u\in\sect(M;\shb_M)$ with ${\WF}(u)\subset\lambda$. Assume that $M$ is connected and that $u\equiv0$ on an open subset $U\subset M$, $U\neq\varnothing$. Then $u\equiv0$ on $M$. 
\end{theorem}
\begin{proof}
Let $S=\supp(u)$ and let $x\in\partial S$. Choosing a local chart in a neighborhood of $x$, we may assume from the beginning that 
$M$ is open in $\R^n$ and that   $\lambda\subset M\times\sqrt{-1}\gamma^\circ$ where 
$\gamma$ is a non empty  open convex cone of $\R^n$. Then there exists a holomorphic function 
$f\in\sect((M\times\sqrt{-1}\gamma)\cap W;\sho_X)$, where $W$ is a connected open neighborhood of $M$ in $X$, such that $u=b(f)$, that is,  
$u$ is the boundary value of $f$. If $b(f)$ is analytic on $U$, then $f$ extends holomorphically in a neighborhood of $U$ in $X$. If moreover $f=0$ on $U$, then $f\equiv0$ on $(M\times\sqrt{-1}\gamma)\cap W$ and thus $u\equiv0$.
\end{proof}

\subsection{D-modules}

Let $(X,\sho_X)$ be a complex manifold. One denotes by  $\shd_X$ the sheaf of rings of finite order holomorphic differential operators on $X$. In the sequel, a $\shd_X$-module means a left $\shd_X$-module. Let $\shm$ be a coherent $\shd_X$-module.
Locally on $X$, $\shm$ may be represented as the cokernel 
of a matrix $\cdot P_0$ of differential operators acting on the right: 
\eqn
&&\shm\simeq \shd_X^{N_0}/\shd_X^{N_1}\cdot P_0
\eneqn
and one shows that 
$\shm$ is locally isomorphic to the cohomology of a bounded complex
\eq\label{eq:globalpresent}
&&\shm^\bullet\eqdot
0\to \shd_X^{N_r}\to\cdots\to\shd_X^{N_1}\to[\cdot P_0]\shd_X^{N_0}\to 0.
\eneq
Clearly, $\sho_X$ is a left $\shd_X$-module. It is indeed coherent since $\sho_X\simeq\shd_X/\shi$ where $\shi$ is the left ideal generated by the vector fields. 
For  a coherent $\shd_X$-module $\shm$, one sets for short 
\eqn
&&\Sol(\shm)\eqdot\rhom[\shd_X](\shm,\sho_X).
\eneqn
Representing (locally) $\shm$ by a bounded complex $\shm^\bullet$ as above, 
we get
\eq\label{eq:solm}
&&\Sol(\shm)\simeq
0\to \sho_X^{N_0}\to[P_0\cdot]\sho_X^{N_1}\to\cdots\sho_X^{N_r}\to 0,
\eneq
where now $P_0\cdot$ operates on the left.

Hence a coherent $\shd_X$-module is nothing but  a system of linear partial differential equations. 

%\subsubsection*{Micro-support and characteristic variety}
To a coherent $\shd_X$-module $\shm$ is associated its characteristic variety $\chv(\shm)$, a closed analytic $\C^\times$-conic co-isotropic subset of $T^*X$. 
\begin{theorem}[{see~\cite{KS90}*{Th.~11.3.3}}]
Let $\shm$ be a coherent $\shd_X$-module. Then $\musupp(\Sol(\shm))=\chv(\shm)$.
\end{theorem}
Let $f\cl Y\to X$ be a morphism of complex manifolds. One can define 
the inverse image $\dopb{f}\shm$, an object of $\Derb(\shd_Y)$.
The Cauchy-Kowalevska theorem has been extended to D-modules in Kashiwara's thesis of 1970. 

\begin{theorem}[{see~\cite{Ka70, Ka03}}]\label{th:CKK}
Let $\shm$ be a coherent $\shd_X$-module and 
assume that $f$ is non characteristic for $\shm$,  that is, for $\chv(\shm)$. Then  
\bnum
\item
$\dopb{f}(\shm)$ is concentrated in degree $0$ and is a coherent $\shd_Y$-module,
\item
$\chv(\dopb{f}(\shm))= f_d\opb{f_\pi}\chv(\shm)$,
\item
one has a natural isomorphism $\opb{f}\rhom[\shd_X](\shm,\sho_X)\isoto\rhom[\shd_Y](\dopb{f}\shm,\sho_Y)$.
\enum
\end{theorem}
\begin{example}
Assume $\shm=\shd_X/\shd_X\cdot P$ for a differential operator $P$ of order $m$ and $Y$ is a smooth hypersurface, non characteristic for $P$. Let $s=0$ be a reduced equation of $Y$. Then, 
 $\dopb{f}(\shm)\simeq\shd_Y/(s\cdot\shd_Y+\shd_X\cdot P)$ and it follows from the Weierstrass division theorem that, locally, 
$\dopb{f}\shm\simeq\shd_Y^m$. In this case, isomorphism (iii) in the above theorem is nothing but the Cauchy-Kowalevska theorem. 
\end{example}

\begin{definition}
Let $\shm$ be a coherent $\shd_X$-module and let $L\subset X$ be a real submanifold. One says that the pair $(L,\shm)$ is elliptic if $
\chv(\shm)\cap T^*_LX\subset T^*_XX$.
\end{definition}
If  $X$ is a complexification of a real  manifold $M$,  the pair $(M,\shm)$ is elliptic if and only if $\shm$ is elliptic in the usual sense and Corollary~\ref{cor:ell} gives, for $F=\Sol(\shm)$, the isomorphism
\eq\label{eq:elliptic1}
&&\rhom[\shd_X](\shm,\sha_M)\isoto\rhom[\shd_X](\shm,\shb_M).
\eneq
In particular, the hyperfunction solutions of the system $\shm$ are real analytic. 
More generally, we have
\begin{theorem}[\cite{Sa70}]
Let $\shm$ be a coherent $\shd_X$-module and 
let $u\in\sect(M;\hom[\shd_X](\shm,\shb_M))$. Then 
$\WF(u)\subset T^*_MX\cap\chv(\shm)$. 
\end{theorem}
When $L=Y$ is a complex submanifold of complex codimension $d$, $(Y,\shm)$ is elliptic if and only if the embedding $Y\into X$ is non-characteristic for $\shm$.  In this case, Corollary~\ref{cor:ell} gives the isomorphism
\eq\label{eq:elliptic2}
&&\opb{f}\rhom[\shd_X](\shm,\sho_X)\isoto\rhom[\shd_X](\shm,\rsect_Y\sho_X)\,[2d].
\eneq

\section{Wick rotation for D-modules}\label{section:wick2}

\subsection{Hyperbolic D-modules}
Let $M$ be a real  manifold and let $X$ be a complexification of $M$. Recall the embedding $T^*M\into T^*T^*_MX$ 
of~\eqref{eq:hyp1} and recall that for  $S\subset T^*X$, the Whitney cone $C_{T^*_MX}(S)$ is contained in $T_{T^*_MX}T^*X\simeq T^*T^*_MX$. 
The next definition is extracted form~\cite{KS90}. See~\cite{Sc13} for details. 
\begin{definition}
Let $\shm$ be a coherent left $\shd_X$-module.
\banum
\item
We set
\eq\label{eq:hyp2}
&&\hchv_M(\shm)=T^*M\cap C_{T^*_MX}(\chv(\shm))
\eneq
and call $\hchv_M(\shm)$ the \em{hyperbolic characteristic variety} of $\shm$ along $M$.
\item
A vector $\theta\in T^*M$ such that $\theta\notin\hchv_M(\shm)$ is called \em{hyperbolic} with respect to $\shm$.
\item
A submanifold $N$ of $M$ is called \em{hyperbolic} for $\shm$ if
\eq\label{eq:hyp3}
&&T^*_NM\cap\hchv_M(\shm)\subset T^*_MM,
\eneq
that is, any nonzero vector of $T^*_NM$ is hyperbolic for $\shm$.
\item
For a differential operator $P$, we set $\hchv(P) = \hchv_M(\shd_X / \shd_X \cdot P)$.
\eanum
\end{definition}
\begin{example}
Assume we have a local coordinate system $(x+\sqrt{-1}y)$ on $X$ with $M=\{y=0\}$ and let $(x+\sqrt{-1}y;\xi+\sqrt{-1}\eta)$ be the coordinates on $T^*X$ so that $T^*_MX=\{y=\xi=0\}$.
Let $(x_0;\theta_0)\in T^*M$ with $\theta_0\not=0$.
Let $P$ be a differential operator with principal symbol $\sigma(P)$.
Then $(x_0;\theta_0)$ is hyperbolic for $P$ if and only if
\eq\label{eq:hyp4}
&&\left\{\parbox{65ex}{
there exist an open neighborhood $U$ of $x_0$ in $M$ and an open conic neighborhood $\gamma$ of $\theta_0\in\R^n$ such that $\sigma(P)(x;\theta+\sqrt{-1}\eta)\not=0$ for all $\eta\in\R^n$, $x\in U$ and $\theta\in\gamma$.
}\right.
\eneq
As noticed by M.~Kashiwara, it follows from the local Bochner's tube theorem that Condition~\eqref{eq:hyp4} can be simplified:
$(x_0;\theta_0)$ is hyperbolic for $P$ if and only if
\eq\label{eq:hyp4K}
&&\left\{\parbox{65ex}{
there exists an open neighborhood $U$ of $x_0$ in $M$ such that $\sigma(P)(x;\theta_0+\sqrt{-1}\eta)\not=0$ for all $\eta\in\R^n$, and $x\in U$.
}\right.
\eneq
Hence, one recovers the classical notion of a (weakly) hyperbolic operator.
\end{example}
\begin{notation}
As usual, we shall 
write $\rhom[\shd_X](\shm,\shc_M)$  instead of $\rhom[\opb{\pi}\shd_X](\opb{\pi}\shm,\shc_M)$  and similarly with other sheaves on cotangent bundles.
\end{notation}

\subsection{Main tool}\label{section:wick2}

Consider as above a real manifold $M$ and a complexification $X$ of $M$,  a closed submanifold $N$ of $M$, 
and $Y$ a complexification of $N$ in $X$. Denote as above by $f\cl Y\into X$ the embedding. 
Consider  also another closed real submanifold $L\subset X$ such that 
$L\cap M=N$ and the intersection  is clean. Denote by $g\cl L\into X$ the embedding and consider the Diagram~\ref{diag:671}.

Let $\shm$ be a coherent $\shd_X$-module and consider the hypotheses:
\eq
&&\parbox{70ex}{the pair $(L,\shm)$ is elliptic, }\label{hyp:1}\\
&&\parbox{70ex}{the submanifold $N$ is hyperbolic for $\shm$ on $M$,}\label{hyp:2}\\
&&\parbox{70ex}{$Y$ is non characteristic for $\shm$}.\label{hyp:3}
\eneq
Set $F=\rhom[\shd_X](\shm,\sho_X)$. Then hypothesis (a)  of Theorem~\ref{th:671} is translated as hypothesis~\eqref{hyp:1}
% that is,  $T^*_LX\cap\,\chv(\shm)\subset T^*_XX$. 
and hypothesis (b) is translated as hypothesis~\eqref{hyp:2}. %, that is,  $T^*_NM\cap\,\hchv_M(\shm)\subset T^*_MM$. 

We shall constantly use the next result.
\begin{lemma}[{see~\cite{JS16}*{Lem.~3.5}}]
Hypothesis~\eqref{hyp:2} implies hypothesis~\eqref{hyp:3}.
\end{lemma}

\begin{theorem} \label{th:671D}
Let $\shm$ be a coherent left $\shd_X$-module. Assume~\eqref{hyp:1} and~\eqref{hyp:2}.
Then one has the natural isomorphism
\eqn
%&&\reim{g_{Nd}}\opb{g_{N\pi}}\rhom[\shd_X](\shm,\shc_M)
%\simeq \mu_N\opb{g}\rhom[\shd_X](\shm,\sho_X)\tens\omega_{Z/N}.
%\eneqn
&&\xymatrix{
\reim{g_{Nd}}\opb{g_{N\pi}}\rhom[\shd_X](\shm,\shc_M)\ar[r]^-\sim%\ar[d]^-\sim
      &\mu_N(\omega_{L/N}\tens\opb{g}\rhom[\shd_X](\shm,\sho_X)).%\ar[d]^-\sim\\
%\roim{g_{Nd}}\epb{g_{N\pi}}\rhom[\shd_X](\shm,\shc_M)&\ar[l]^-\sim\mu_N(\omega_{X/N}\tens\epb{g}\rhom[\shd_X](\shm,\sho_X)).
}\eneqn

\end{theorem}
\begin{proof}
Apply Theorem~\ref{th:671} together with Lemma~\ref{le:671} to the sheaf   $F=\rhom[\shd_X](\shm,\sho_X)$.
%The isomorphism  is obtained by passing $\omega_X^{\tens-1}$ on the right to $\omega_X$ on the left and by passing 
%$\omega_N$ on the left to $\omega_N^{\tens-1}$ on the right. 
 We get:
\eqn
&&\reim{g_{Nd}}(\omega_{N/M}\tens\opb{g_{N\pi}}\rhom[\shd_X](\shm,\mu_M(\sho_X)))
\simeq \mu_N(\omega_{L/X}\tens\opb{g}\rhom[\shd_X](\shm,\sho_X)).
\eneqn
Equivalently, we have
\eqn 
&&\reim{g_{Nd}}\opb{g_{N\pi}}(\omega_{X/M}\tens\rhom[\shd_X](\shm,\mu_M(\sho_X)))
\simeq \mu_N(\omega_{L/N}\tens\opb{g}\rhom[\shd_X](\shm,\sho_X)).
\eneqn
Finally $\omega_{X/M}\tens\mu_M(\sho_X)\simeq\shc_M$.
\end{proof}

\subsubsection*{Example 1: Cauchy problem for microfunctions}
Let $M$, $X$, $L$, $N$ and $f$ be as above and assume that 
 $L=Y$, hence $f=g$. 

\begin{corollary} \label{cor:671D1}
Let $\shm$ be a coherent left $\shd_X$-module. Assume~\eqref{hyp:2}.
Then one has the natural isomorphism
\eqn
&&\eim{f_{Nd}}\opb{f_{N\pi}}\rhom[\shd_X](\shm,\shc_M)\simeq \rhom[\shd_Y](\dopb{f}\shm,\shc_N).
\eneqn
\end{corollary}
\begin{proof}
Applying Theorem~\ref{th:CKK}, 
we get $\opb{f}\rhom[\shd_X](\shm,\sho_X)\simeq\rhom[\shd_Y](\dopb{f}\shm,\sho_Y)$. (Recall that~\eqref{hyp:2} implies~\eqref{hyp:3}.)
Moreover, $\omega_{Y/N}\tens\mu_N(\sho_Y)\simeq\shc_N$. Finally, since $f_{Nd}$ is finite on $\chv(\shm)$, we may replace $\reim{f_{Nd}}$ with $\eim{f_{Nd}}$. 
\end{proof}

\subsection{Boundary values}\label{subsection:bv}
Let $M$ be a real $n$-dimensional manifold, $N$ a closed submanifod of codimesnion $d$,  $X$ a complexification of $M$ and $Y$ a complexification of $N$ in $X$.  We denote by $f\cl Y\into X$ the embedding. 
%Let also $L$ be another real manifold such that $X$ is a complexification of $L$ and 
%$N=L\cap M$, the intersection being clean. 
%We shall apply Theorem~\ref{th:671D} with the new notation $L=Z$.

\begin{notation}
We set
\eqn
&&\tw\shb_N=\rsect_N(\sho_X)\tens\ori_N\,[n]\simeq H^n_N(\sho_X)\tens\ori_N.
\eneqn
\end{notation}
We shall not confuse the sheaf $\tw\shb_N$  with the sheaf $\shb_N$ of hyperfunctions  on $N$.
We have an isomorphism
\eqn
&&\tw\shb_N\simeq \sect_N\shb_M\tens\ori_{N/M}\simeq  \sect_N\shb_M\tens\omega_{M/N}\,[-d].
\eneqn
Let $\shm$ be a coherent $\shd_X$-module. 
Applying the functor $\rsect_N(\scbul)\tens\ori_N\,[n-d]$ to the isomorphism (iii)  in Theorem~\ref{th:CKK} together with isomorphism~\eqref{eq:elliptic2} one recovers a well known result:
\begin{lemma}
Assume~\eqref{hyp:3}.
One has a natural isomorphism 
\eqn
&&\rhom[\shd_X](\shm,\tw\shb_N)\,[d]\simeq \rhom[\shd_Y](\dopb{f}\shm,\shb_N). 
\eneqn
\end{lemma}
Appying the functor $\RD'_X$ to the morphism $\C_M\to\C_N$, we get the morphism $\RD'_X(\C_N)\to\RD'_X(\C_M)$, that is, the morphism $\ori_N\,[d+n]\to \ori_M\,[n]$. 
Applying the functor $\rhom(\scbul,\sho_X)$ we get the ``restriction''  morphism
\eq\label{eq:bv1}
&&\rho_{MN}\cl \shb_M\to \tw\shb_N\,[d]\simeq\sect_N\shb_M\tens\omega_{M/N}.
\eneq
For a closed cone $\lambda\subset T^*_MX$, we set for short
\eq\label{eq:sectmu}
&&\shb_{M,\lambda}\eqdot\oim{\pi_M}\sect_{\lambda}\shc_M.
\eneq
For an open cone $\gamma\subset T_NM$, we set for short :
\eq\label{eq:sectnu}
&&\sect_\gamma\shb_{NM}\eqdot\oim{\tau_N}\sect_\gamma(\nu_N(\shb_M)).
\eneq
(In the sequel, we shall use this notation for another real manifold $L$ instead of $M$.)

Hence, for a closed convex proper cone $\delta\subset T^*_NM$ with $\delta\supset N$, setting $\gamma=\Int(\delta^{\circ a})$,  we have by~\eqref{eq:nuandmu}:
\eq\label{eq:nuandmu2}
&&\oim{\pi_N}\sect_{\delta}(\mu_N\shb_M)\tens\omega_{M/N}\simeq\sect_{\gamma}\shb_{NM}. 
\eneq
One can use~\eqref{eq:nuandmu2} and the morphism $\oim{\pi_N}\sect_{\delta}(\mu_N\shb_M)\to\oim{\pi_N}\mu_N\shb_M\simeq \sect_N\shb_M$ to obtain the morphism
\eq\label{eq:bv2}
&&b_{\gamma, N}\cl \sect_\gamma\shb_{NM}\to\sect_N\shb_M\tens\omega_{M/N}.
\eneq
One can also construct~\eqref{eq:bv2} directly as follows. Let $U$ be an open subset of $M$ such that $\ol U\supset N$, $U$ is 
locally cohomologically trivial (see~\cite{KS90}*{Exe.~III.4}). Then 
the morphism $\C_{\ol U}\to\C_N$ gives by duality the morphism $\ori_N\,[d+n]\to \ori_U\,[n]$ and one gets the morphism 
$\sect_U\shb_M\to \sect_N\shb_M\tens\omega_{M/N}$ by applying  $\rhom(\scbul,\sho_X)$ similarly as for $\rho_{MN}$.  
Taking the inductive limit with respect to the family of open sets $U$ such that $C_M(X\setminus U)\cap \gamma=\varnothing$
(see~\cite{KS90}*{Th.~4.2.3}), we recover the morphism~\eqref{eq:bv2}. 

In particular, for a coherent $\shd_X$-module $\shm$
we get the morphisms
\eqn
\rho_{MN}&\cl&\rhom[\shd_X](\shm,\shb_{M,\lambda})\to\rhom[\shd_X](\shm,\tw\shb_N)\,[d],\\
b_{\gamma, N}&\cl& \rhom[\shd_X](\shm,\sect_\gamma\shb_{NM})\to\rhom[\shd_X](\shm,\tw\shb_N)\,[d].
\eneqn

\subsection{Wick rotation}

Let $M$, $X$, $Y$, $N$, $L$, $f$ and $g$ be as above. Now, we also  assume that $L$ is a real manifold of the same dimension 
than $M$ and  $X$ is a complexification of $L$. We still consider diagram~\eqref{diag:671}.

Consider the hypothesis
\eq\label{hyp:4}
&&\left\{ \parbox{70 ex}{in a neighborhood of $N$, $\chv(\shm)\cap T^*_MX$ is contained in the union of two closed cones 
$\lambda$ and $\lambda^{\prime}$ such that $\lambda\cap\lambda^{\prime}=M\times_XT^*_XX.$}\right.
\eneq
(Here,  $M\times_XT^*_XX$ stands for the zero-section of $T^*_MX$.)
\begin{lemma}\label{le:4b}
Assume~\eqref{hyp:4}.Then we have the natural isomorphism
\eq\label{eq:4b}
&&\opb{g_{N\pi}}\rsect_{\lambda}\rhom[\shd_X](\shm,\shc_M)\isoto\rsect_{\opb{g_{N\pi}}(\lambda)}\opb{g_{N\pi}}\rhom[\shd_X](\shm,\shc_M).
\eneq
\end{lemma}
\begin{proof}
(i) Set for short $F=\rhom[\shd_X](\shm,\shc_M)$, $j=g_{N\pi}$ , $A=\lambda$, $B=\opb{j}A$. With these new notations, we have to prove the morphism 
\eq\label{eq:4c}
&&\opb{j}\rsect_AF\isoto\rsect_B\opb{j}F
\eneq
is an isomorphism. 

\spa
(ii) The morphism~\eqref{eq:4c} is an isomorphism outside of the zero-section of $T^*_MX$ since then 
$\supp(F)=A\sqcup C$ with $A$ and $C$ closed and $A\cap C=\varnothing$, by the hypothesis~\eqref{hyp:4}.

\spa
(iii) Consider the diagram in which $s_N$ and $s_M$ denote the embeddings of the zero-sections:
\eq\label{diag:2}
&&\xymatrix{
N\times_MT^*_MX\ar@<-0.5ex>[d]_-{\pi_N}\ar[r]^-j&T^*_MX\ar@<-0.5ex>[d]_-{\pi_M}\\
N\ar[r]^-j\ar@<-0.5ex>[u]_-{s_N}&M\ar@<-0.5ex>[u]_-{s_M}.
}\eneq
Since $\roim{\pi_N}\simeq\opb{s_N}$, when applied to conic sheaves, it remains to show that~\eqref{eq:4c} is an isomorphism after applying the functor $\roim{\pi_N}$.

\spa
(iv) Consider the morphism of Sato's distinguished triangles:
\eqn
&&\xymatrix{
\reim{\pi_N}\opb{j}\rsect_AF\ar[r]\ar[d]^-u&\roim{\pi_N}\opb{j}\rsect_AF\ar[r]\ar[d]^-v
              &{\dot\pi}_{N*}\opb{j} \rsect_AF\ar[r]^-{+1}\ar[d]^-w&\\
\reim{\pi_N}\rsect_B\opb{j}F\ar[r]&\roim{\pi_N}\rsect_B\opb{j}F\ar[r]&{\dot\pi}_{N*}\rsect_B\opb{j}F\ar[r]^-{+1}&
}\eneqn
It follows from (ii) that the vertical arrow $w$ on the right is an isomorphism. We are thus reduced to prove the isomorphism
\eq\label{eq:4d}
&&\reim{\pi_N}\opb{j}\rsect_AF\isoto\reim{\pi_N}\rsect_B\opb{j}F.
\eneq

\spa
(v) Using the fact that $A\supset M$ and $B\supset N$ and that Diagram~\eqref{diag:2} with the arrows going down  is Cartesian, we get
\eqn
\reim{\pi_N}\opb{j}\rsect_AF&\simeq&\opb{j}\reim{\pi_M}\rsect_AF\simeq\opb{j}\epb{s_M}\rsect_AF\\
&\simeq&\opb{j}\epb{s_M}F\simeq \opb{j}\reim{\pi_M}F\simeq \reim{\pi_N}\opb{j}F\\
&\simeq&\epb{s_N}\opb{j}F\simeq \epb{s_M}\rsect_B\opb{j}F \\
&\simeq&\reim{\pi_N}\rsect_B\opb{j}F .
\eneqn
\end{proof}
Consider 
\eq\label{eq:conegamma}
&&\gamma\subset T_NL\mbox{ an open convex cone such that $\ol\gamma$ contains the zero-section $N$}
\eneq
and recall notations~\eqref{eq:sectmu} and~\eqref{eq:sectnu}.

\begin{theorem}[{Wick isomorphism Theorem}] \label{th:wick}
Let $\shm$ be a coherent left $\shd_X$-module and let $\gamma$ be as in~\eqref{eq:conegamma}. 
Assume~\eqref{hyp:1},~\eqref{hyp:2},~\eqref{hyp:4} and also
\eq\label{eq:hyp:5}
&&\opb{g_{N\pi}}(\lambda)=\opb{g_{Nd}}(\gamma^{\circ a}).
\eneq
Then one has the commutative diagram   in which the horizontal arrow is an isomorphism:
\eq\label{eq:wick2}
&&\xymatrix@C=3ex@R=3.5ex{
\rhom[\shd_X](\shm,\shb_{M,\lambda})\vert_N\ar[rr]^-\sim\ar[dr]_-{\rho_{MN}} 
       &&\rhom[\shd_X](\shm,\sect_\gamma\shb_{NL})\ar[ld]^-{b_{\gamma,N}}\\
       &\rhom[\shd_X](\shm,\tw\shb_N)\,[d]&
}\eneq
\end{theorem}
\begin{proof}
(i) As a particular case of Theorem~\ref{th:671D} and using the fact that 
$\opb{g}\rhom[\shd_X](\shm,\sho_X)\simeq\rhom[\shd_X](\shm,\shb_L)$, we get the  isomorphism
\eqn
&&\reim{g_{Nd}}\opb{g_{N\pi}}\rhom[\shd_X](\shm,\shc_M)\simeq \rhom[\shd_X](\shm,\mu_N\shb_L)\tens\omega_{L/N}.
\eneqn

\spa
(ii)  Set  for short $F=\rhom[\shd_X](\shm,\shc_M)$. Using Lemma~\ref{le:4b} and the fact that $g_{Nd}$ is proper on $\supp F$, we have the isomorphism
\eqn
\reim{g_{Nd}}\opb{g_{N\pi}}\rsect_{\lambda}F&\simeq&\reim{g_{Nd}} \rsect_{\opb{g_{N\pi}}(\lambda)}\opb{g_{N\pi}}F\\
&\simeq&\rsect_{\gamma^{\circ a}} \reim{g_{Nd}}\opb{g_{N\pi}}F.
\eneqn
Therefore, we have proved the isomorphism
\eq\label{eq:wick3}
&&\reim{g_{Nd}}\opb{g_{N\pi}}\rhom[\shd_X](\shm,\sect_{\lambda}\shc_M)
  \simeq \rhom[\shd_X](\shm,\sect_{\gamma^{\circ a}}\mu_N\shb_L)\tens\omega_{L/N}.
\eneq

\spa
(iii) Let us apply the functor $\roim{\pi_N}$ to~\eqref{eq:wick3}.  Since $g_{Nd}$ is proper on $\supp F$, 
setting $G= \rhom[\shd_X](\shm,\sect_{\lambda}\shc_M)$, we have  (see Diagram~\ref{diag:5})
\eqn
\roim{\pi_N}\reim{g_{Nd}}\opb{g_{N\pi}}G&\simeq& \roim{\pi}\opb{g_{N\pi}}G\\
 &\simeq& (\roim{\pi_M}G)\vert_N.
 \eneqn
 Hence, we have proved the isomorphism
\eqn
&&\rhom[\shd_X](\shm,\shb_{M,\lambda})\vert_N\simeq \rhom[\shd_X](\shm,\oim{\pi_N}\sect_{\gamma^{\circ a}}\mu_N\shb_L)\tens\omega_{L/N}
\eneqn
and the result follows from~\eqref{eq:nuandmu2}.
\end{proof}

\subsection{The classical Wick rotation}
Let us treat the classical Wick rotation. Hence, we assume that
$M=N\times\R$ and $L=N\times\sqrt{-1}\R$. As usual, $Y$ is a complexification of $N$ and $X=Y\times\C$. 
We denote by  $t+is$ the holomorphic coordinate on $\C$,  by $(t+is;\tau+i\sigma)$ the symplectic coordinates on $T^*\C$ and by $(x;i\eta)$ a point of $T^*_NY$. We identify $N$ and $N\times\{0\}\subset X$.

Let $P$ is a differential operator of order $m$, elliptic on $L$ and (weakly) hyperbolic on $M$ in the $\pm dt$ codirections. A typical example 
is the wave operator on a globally hyperbolic spacetime $N\times\R_t$.  Set
\eqn
&&L^+=N\times \{t+is;t=0, s>0\},\quad \lambda=T^*_NY\times\{(t+is;\tau+i\sigma);s=0,\tau=0,\sigma\leq0\}.
\eneqn
The map $g_{Nd}\cl N\times_MT^*_MX\to T^*_NL$ is given by
\eqn
&& (x,0;i\eta,i\sigma)\mapsto (x;\sigma).
\eneqn
We shall apply the preceding result with $\gamma=L^+$. In that case, $\gamma^{\circ a}= \lambda$
and~\eqref{eq:hyp:5} is satisfied.

Let $\shm=\shd_X/\shd_X\cdot P$. In the sequel we write for short $\shb_M^P$ instead of $\hom[\shd_X](\shm,\shb_M)$ and similarly with other sheaves. Note that $\ext[\shd_X]{1}(\shm,\tw\shb_N)\simeq \tw\shb_N/P\cdot\tw\shb_N$.

As a particular case of Theorem~\ref{th:wick}, we get:
\begin{corollary}\label{cor:wick}
We have a commutative diagram in which the horizontal arrow is an isomorphism:
\eqn
&&\xymatrix@C=3ex@R=3.5ex{
\shb^P_{M,\lambda}\vert_N\ar[dr]_-{\rho}\ar[rr]^-\sim &&\shb^P_{L^+}\vert_N\ar[ld]^-{b}\\
&\tw\shb_N/P\cdot\tw\shb_N\ar[r]_-\sim&\shb_N^m.
}\eneqn
\end{corollary}

\providecommand{\bysame}{\leavevmode\hbox to3em{\hrulefill}\thinspace}

\vspace*{1cm}
\noindent
\parbox[t]{21em}
{\scriptsize{
Pierre Schapira\\
Sorbonne Universit{\'e}s, UPMC Univ Paris 06\\
Institut de Math{\'e}matiques de Jussieu\\
e-mail: pierre.schapira@imj-prg.fr\\
http://webusers.imj-prg.fr/\textasciitilde pierre.schapira/
}}
\end{document}